\documentclass[a4paper,12pt]{article}
\usepackage[top=3cm,bottom=2.5cm,left=2.5cm,right=2.5cm]{geometry}
\begin{document}

\thispagestyle{empty}

\newtheorem{theorem}{Theorem}[section]
\newtheorem{corollary}[theorem]{Corollary}
\newtheorem{definition}[theorem]{Definition}
\newtheorem{conjecture}[theorem]{Conjecture}
\newtheorem{question}[theorem]{Question}
\newtheorem{lemma}[theorem]{Lemma}
\newtheorem{proposition}[theorem]{Proposition}
\newtheorem{example}[theorem]{Example}
\newtheorem{problem}[theorem]{Problem}
\newenvironment{proof}{\noindent {\bf
Proof.}}{\rule{3mm}{3mm}\par\medskip}
\newcommand{\remark}{\medskip\par\noindent {\bf Remark.~~}}
\newcommand{\pp}{{\it p.}}
\newcommand{\de}{\em}

\newcommand{\JEC}{{\it Europ. J. Combinatorics},  }
\newcommand{\JCTB}{{\it J. Combin. Theory Ser. B.}, }
\newcommand{\JCT}{{\it J. Combin. Theory}, }
\newcommand{\JGT}{{\it J. Graph Theory}, }
\newcommand{\ComHung}{{\it Combinatorica}, }
\newcommand{\DM}{{\it Discrete Math.}, }
\newcommand{\ARS}{{\it Ars Combin.}, }
\newcommand{\SIAMDM}{{\it SIAM J. Discrete Math.}, }
\newcommand{\SIAMADM}{{\it SIAM J. Algebraic Discrete Methods}, }
\newcommand{\SIAMC}{{\it SIAM J. Comput.}, }
\newcommand{\ConAMS}{{\it Contemp. Math. AMS}, }
\newcommand{\TransAMS}{{\it Trans. Amer. Math. Soc.}, }
\newcommand{\AnDM}{{\it Ann. Discrete Math.}, }
\newcommand{\NBS}{{\it J. Res. Nat. Bur. Standards} {\rm B}, }
\newcommand{\ConNum}{{\it Congr. Numer.}, }
\newcommand{\CJM}{{\it Canad. J. Math.}, }
\newcommand{\JLMS}{{\it J. London Math. Soc.}, }
\newcommand{\PLMS}{{\it Proc. London Math. Soc.}, }
\newcommand{\PAMS}{{\it Proc. Amer. Math. Soc.}, }
\newcommand{\JCMCC}{{\it J. Combin. Math. Combin. Comput.}, }
\newcommand{\GC}{{\it Graphs Combin.}, }
\thispagestyle{empty}
\title{On the Two Conjectures of the Wiener Index \thanks{
 Supported by National Natural Science Foundation of China
(No. 11271256). \newline \indent
$^{\dagger}$Correspondent author: Xiao-Dong Zhang (Email:
xiaodong@sjtu.edu.cn)}}
\author{  Ya-Lei Jin  and Xiao-Dong Zhang$^{\dagger}$   \\
{\small Department of Mathematics, and Ministry of Education }\\
{\small Key Laboratory of Scientific and Engineering Computing, }\\
{\small Shanghai Jiao Tong University} \\
{\small  800 Dongchuan road, Shanghai, 200240, P.R. China}
}
\date{}
%\date{(Received February 19, 2013)}
\maketitle
\begin{center}
(Received February 19, 2013)
\end{center}

{\bf Abstract}\\
\baselineskip=0.30in

 \begin{minipage}{5in}
 The Wiener index of a graph, which is the sum of the distances between all pairs of vertices, has been well studied. Recently, Sills and Wang in 2012 proposed two conjectures on the maximal Wiener index of trees with a given degree sequence. This note proves one of the two conjectures and disproves the other.
 \end{minipage}

%\vskip 0.5cm

\section{Introduction}

The Wiener index of  a  molecular graph is one of the most classic and well-known topological indices in the molecular graph, which was introduced by  and named by Wiener \cite{wiener1947} in 1947. It has been extensively studied  by chemists and mathematicians over the past years, see for instance \cite{cela2011}.
 In the past decade years, the extremal trees that maximize or minimize  the Wiener index among trees with prescribed maximum degree, diameter, matching and independence numbers, etc., have been studied (see \cite{fischermann2002,liu2008, Zhang2008} etc.).

Since the degrees of a molecular graph corresponds to the valences of the atoms, it is one of  the most interesting aspects to consider all trees with a prescribed degree sequence.
Wang \cite{wang2008} and Zhang at al. \cite{Zhang2008} independently proved the extremal tree that minimizes the Wiener index is  greedy tree through different approaches.
Moreover, the extremal tree that maximizes the Wiener index in this category in \cite{wang2008} is incorrect by pointed out in  \cite{wang2010} and \cite{zhang2010}. Therefore it is still open problem.
 \begin{problem}
   Characterize the extremal trees that maximize the Wiener index with prescribed degree sequence.
   \end{problem}
     Zhang et al. \cite{zhang2010} provided some part results with less than 7 internal vertices.  Cela  et al. \cite{cela2011} provide an efficient algorithm for finding the extremal trees with prescribed degree sequence. Recently, Sills and Wang \cite{sills2012} further studied the maximal Wiener index and disclosed some relations between the candidate trees for the maximal Wiener index and the symmetric Dyck paths.

 Let $T=(V, E)$ be a tree of order $n$. The {\it Wiener index} $W(T)$ of $T$ is defined as
 $$W(T):=\sum_{\{u,v\}\subseteq V}d(u,v),$$
 where $d(u, v)$ is the number of edges in a shortest path from $u$ to $v$.
 A nonincreasing
sequence of nonnegative integers
 $\pi=(d_1,d_2,\cdots, d_{n})$ is called {\it graphic} if there
 exists a  simple graph having $\pi$ as its  vertex degree sequence.
 In particular, if $\sum_{i=1}^nd_i=2(n-1)$, then $\pi$ is graphic and any graph with degree sequence $\pi$ is tree and let $\mathcal{T}_{\pi}$ denote the set of all trees with degree sequence $\pi$. Moreover, if
 $$d_1\ge d_2\ge \dots\ge d_k\ge 2>d_{k+1}=d_{k+2}=\dots =d_n=1,$$
then $b=(b_1,\dots,b_k):=(d_1-1,\dots,d_k-1)$ is called the {\it decremented degree sequence}\cite{sills2012}.  A {\it caterpillar}
is a tree in which a single path (called {\it Spine}) is incident to
(or contains) every edge. For other terminologies and notions, we follow from \cite{bondy1976, sills2012}.
 Since it has been proved \cite{zhang2010} that a tree with maximum Wiener index in
 $\mathcal{T}_{\pi}$ has to be a caterpillar, it is interesting  and important to study
  the Wiener index of caterpillars.  Let $T$ be a caterpillar of order $n$ with $n-k$ leaves and the non-leaf vertices $v_1,\dots, v_k$. Then the Winer index of $T$ is presented in \cite{zhang2010}
  $$W(T)=(n-1)^2+q(x),$$
  where $q(x) $is the quadratic form
  \begin{equation}\label{eqn1}
  q(x)=\frac{1}{2}\sum_{i=1}^k\sum_{j=1}^k|i-j|x_ix_j=x^TA_kx,
  \end{equation}
 $A_k=(a_{ij})$ with $a_{ij}=\frac{1}{2}|i-j|$, $x=(x_1, \dots,x_k)^T$, and $x_i=deg(v_i)-1$ for $i=1, \dots, k$.
In order to obtain some useful upper bounds for the Wiener index in $\mathcal{T}_{\pi}$,  Sills and Wang observed the largest eigenvalue of $A_k$ is about to
\begin{equation}
\lambda_{\max}\approx\frac{\sqrt{3}k^2-2}{10}.\end{equation}
 Further, they disclosed some interesting  combinatorial relations to other objects from this study and
proposed the following conjecture.
 \begin{conjecture}\cite{sills2012}\label{con1}
 Let $A_k=(a_{ij})$ be the $k\times k$ matrix with $a_{ij}=\frac{1}{2}|i-j|$. If
 $C_k(\lambda)=det(A_k-\lambda I_k)$ is the characteristic polynomial of $A_k$,
  then
  \begin{equation}\label{ch3}
  C_k(\lambda)=(-1)^k\lambda^k\left(1-\frac{k}{4}\sum_{j=1}^{k-1}\frac{j}{j+1}
  \left(\begin{array}{c} k+j\\ 2j+1\end{array}\right)\lambda^{-j-1}\right).
  \end{equation}
  \end{conjecture}
On the other hand,  Silly and Wang \cite{sills2012} characterized all extremal trees that maximize in all  chemical trees with prescribed degree sequence $\pi=(d_1,\dots, d_n)$ with $4\ge d_1\ge\dots\ge d_n=1$. This result can be stated as follows:
  \begin{theorem}\cite{sills2012}\label{chemical}
  Let $\pi=(d_1,\dots d_k, d_{k+1},\dots, d_n)$ with $4\ge d_1\ge\dots\ge d_k>d_{k+1}=\dots=d_n=1$ and let $b=(b_1, \dots, b_k)$ be the decrmented degree sequence. If $\{b_1,b_2,\dots, b_k\}=\{ \underbrace{a_s,\dots,a_s}_{m_s}, \underbrace{
  a_{s-1},\dots,a_{s-1}}_{m_{s-1}}, \dots, \underbrace{a_1,\dots, a_1}_{m_1}\}$ with
  $a_s>a_{s-1}>\dots>a_1,$  then
  $q(x)$ is maximized by
  $$x=\{  \underbrace{a_s,\dots,a_s}_{l_s}, \underbrace{
  a_{s-1},\dots,a_{s-1}}_{l_{s-1}}, \dots, \underbrace{a_1,\dots, a_1}_{m_1},
  \dots, , \underbrace{
  a_{s-1},\dots,a_{s-1}}_{r_{s-1}}, \underbrace{a_s,\dots,a_s}_{r_s}
  \},$$
  where $|l_i-r_i|\le 1$ and $l_i+r_i=m_i$ for $i=2, \dots, s.$
  \end{theorem}

  Further, they \cite{sills2012} proposed the following conjecture
  \begin{conjecture}\cite{sills2012}\label{con2}
  When $k$ is much larger than $s$, for
   $$\{b_1,b_2,\dots, b_k\}=\{ \underbrace{a_s,\dots,a_s}_{m_s}, \underbrace{
  a_{s-1},\dots,a_{s-1}}_{m_{s-1}}, \dots, \underbrace{a_1,\dots, a_1}_{m_1}\}$$
  with
  $a_s>a_{s-1}>\dots>a_1,$  then
  $q(x)$ is maximized by
  $$x=\{  \underbrace{a_s,\dots,a_s}_{l_s}, \underbrace{
  a_{s-1},\dots,a_{s-1}}_{l_{s-1}}, \dots, \underbrace{a_1,\dots, a_1}_{m_1},
  \dots, , \underbrace{
  a_{s-1},\dots,a_{s-1}}_{r_{s-1}}, \underbrace{a_s,\dots,a_s}_{r_s}
  \},$$
  where $|l_i-r_i|\le 1$ and $l_i+r_i=m_i$ for $i=2, \dots, s$.
  \end{conjecture}

  This note is motivated by the above two conjectures. The rest of the note is organized as follows: In next Section, we prove  Conjecture~\ref{con1}; while in Section 3, we disprove  Conjecture~\ref{con2}.
  \section{Proof of Conjecture~\ref{con1}}

  Before presenting a proof of Conjecture~\ref{con1}, we need some notations. Let $G=(V, E)$ be a connected graph with $V=\{v_1, \dots, v_n\}$, Graham and Pollak \cite{graham1971} introduced the {\it distance matrix } $D(G)=(d_{ij})$ of $G$ with $d_{ij}=d(v_i, v_j)$  arising from a data communication problem. Graham and Lov\'{a}sz \cite{graham1978} proved that the coefficients of the characteristic polynomial of the distance matrix of a tree can be expressed in terms of the number of certain subforests of the tree and conjectured that the sequence of coefficients was  unimodal with peak at the center. Colllins \cite{collins1989} proved that the coefficients for a path on $n$ vertices are unimodal with peak at $(1-1/\sqrt{5})n$. From the context, it is easy to get the following Lemma from \cite{collins1989}
  \begin{lemma}\cite{collins1989}\label{lemma1}
  Let $P_n$ be a path of order $n$ and distance matrix $D(P_n)=(d_{ij})$ with $d_{ij}=|i-j|$, i.e.,
  $$D(P_n)=\left(\begin{array}{ccccc}
  0 & 1 & 2 & \dots & n-1\\
  1 & 0 & 1& \dots & n-2\\
  \dots & \dots & \dots & \dots & \dots\\
  n-1& n-2&n-3 & \dots & 0\end{array}\right).$$
  let $\delta_i$ be the coefficient of $\lambda^i$ in the distance matrix polynomial $\det(D(P_n)-\lambda I_n)$. Then
  $$ \delta_n=(-1)^n, \ \delta_{n-i}=(-1)^{n-1}\frac{2^{i-2}n(i-1)}{i}\left(\begin{array}{c}
  n+i-1\\2i-1\end{array}\right), \ \ \ \ for \ \ i=1, \dots, n.$$
  \end{lemma}
  \begin{proof} It follows from \cite{collins1989}.\end{proof}
  Now we are ready to prove Conjecture~\ref{con1}
  \begin{theorem}\label{thm1}
  Let $A_k=(a_{ij})$ be the $k\times k$ matrix with $a_{ij}=\frac{1}{2}|i-j|$. If
 $C_k(\lambda)=det(A_k-\lambda I_k)$ is the characteristic polynomial of $A_k$,
  then
  $$C_k(\lambda)=(-1)^k\lambda^k\left(1-\frac{k}{4}\sum_{j=1}^{k-1}\frac{j}{j+1}
  \left(\begin{array}{c} k+j\\ 2j+1\end{array}\right)\lambda^{-j-1}\right).$$
    \end{theorem}
 \begin{proof}
 Clearly, $A_k=\frac{1}{2}D(P_k)$.  Then by Lemma~\ref{lemma1}
 \begin{eqnarray*}
 C_k(\lambda)&=&\det(A_k-\lambda I)=\det(\frac{1}{2}D(P_k)-\lambda I)\\
 &=& (\frac{1}{2})^k\det(D(P_k)-(2\lambda) I_k)\\
 &=& (\frac{1}{2})^k ( (-1)^k(2\lambda)^k+\dots+\delta_{n-i}(2\lambda)^{n-i}+\dots+\delta_0)\\
 &=&(-1)^k\lambda^k+\dots+\frac{(-1)^{k-1}(i-1)k}{4i}\left(\begin{array}{c}
 k+i-1\\2i-1\end{array}\right)\lambda^{k-i}+\dots+\frac{(-1)^{k-1}(k-1)}{4}\\
 &=&(-1)^k\lambda^k\left(1-\frac{k}{4}\ \sum_{j=1}^{k-1}\frac{j}{j+1}
 \left(\begin{array}{c}
 k+j\\2j+1\end{array}\right)\lambda^{-j-1}\right).
 \end{eqnarray*}
Hence Theorem~\ref{thm1} holds. \end{proof}
On the largest eigenvalue of $A_k$, there is the following result.
\begin{theorem}\label{thm1-2}
The largest eigenvalue of $A_k=(a_{ij})$ with $a_{ij}=\frac{1}{2}|i-j|$ is equal to
$$\lambda_{\max}=\frac{1}{2(\cosh \theta-1)}, $$
 where $\theta$ is the positive solution of  \ $\tanh (\frac{\theta}{2})\tanh(\frac{k\theta}{2})=\frac{1}{k}$. Moreover,
 $$\lambda_{\max}=\frac{k^2}{4a^2}-\frac{2+a^2}{12a^2}+o(\frac{1}{n^2}),$$
 where $a$ is the root of a $\tanh(a)=1$, i.e, $a\approx 1.199679.$
 \end{theorem}
 \begin{proof}
 It follows from $A_k=\frac{1}{2}D(P_k)$, Theorem~2.1 and Corollary~2.2 in \cite{ruzoeh1990}
 \end{proof}

\section{Disproof of Conjecture~\ref{con2}}
In order to disprove Conjecture~\ref{con2}, we first present the following result
\begin{theorem}\label{thm2}
Let  $\{b_1,b_2,\dots, b_k\}=\{ \underbrace{a_s,\dots,a_s}_{m_s}, \underbrace{
  a_{s-1},\dots,a_{s-1}}_{m_{s-1}}, \dots, \underbrace{a_1,\dots, a_1}_{m_1}\}$
  with
  $a_s>a_{s-1}>\dots>a_1.$ If $a_s>\sum_{i=1}^{s-1}m_ia_i $ and $m_s=2h+1$,  then
  $q(x)$ is uniquely maximized by
  $$x=\{  \underbrace{a_s,\dots,a_s}_{h+1}, \underbrace{a_1,\dots, a_1}_{m_1},
  \dots,  \underbrace{
  a_{s-1},\dots,a_{s-1}}_{r_{s-1}}, \underbrace{a_s,\dots,a_s}_{h}
  \}.$$
\end{theorem}
\begin{proof}
It is easy to see that the assertion hold for $s=2$ or $k\le 5.$ Now assume that
  $q(x)$ is maximized by $x=\{x_1, \dots, x_k\}$ for $k>5$ and $s\ge 3$. Then by Theorem~2.7 in \cite{zhang2010}, there exists a  $2\le t\le k-2$ such that
  \begin{equation}\label{eq1}
  \sum_{i=1}^{t-2}x_i\le \sum_{i=t+1}^kx_i, \ \  \sum_{i=1}^{t-1}x_i> \sum_{i=t+2}^kx_i,
  \end{equation}
  and
  either $x_1\ge \dots\ge  x_t, \ x_t\le x_{t+1}\le\dots\le x_k$  or  $x_1\ge \dots\ge  x_{t-1}, \ x_{t-1}\le x_{t+1}\le\dots\le x_k$.
  Hence $x$ can be rewritten as
  $$x=\{\underbrace{a_s,a_s,...,a_s}_{l_s},\underbrace{a_{s-1},a_{s-1},...,a_{s-1}}_{l_{s-1}},...,
\underbrace{a_1,a_1,...,a_1}_{m_1},....,\underbrace{a_{s-1},...,a_{s-1}}_{r_{s-1}},\underbrace{a_s,a_s,...,a_s}_{r_s}\},$$
  where $l_s+r_s=m_s=2h+1$.
  Clearly $t> l_s$ and $l_s>r_s$. Further we have the following \\
  {\bf claim } $t=l_s+1$ and $ l_s=r_s+1$.
  In fact, suppose that $t\ge l_s+2.$ Then by the condition of Theorem~\ref{thm2},
  $$ \sum_{i=1}^{t-2}x_i\ge l_sa_s\ge(r_s+1)a_s>r_aa_s+\sum_{i=1}^{s-1}m_ia_i\ge \sum_{i=t+1}^kx_i,$$
  which is contradiction to (\ref{eq1}).  Hence $t=l_s+1$.
  Moreover, by (\ref{eq1}), we have
  $$(l_s-1)a_s=\sum_{i=1}^{t-2}x_i<\sum_{i=t+1}^kx_i\le r_sa_s+\sum_{i=1}^{s-1}m_ia_i<(r_s+1)a_s.$$
So $l_s-1<r_s+1$, i.e., $l_s=r_s+1$.
Therefore the assertion holds.
\end{proof}
{\bf Remark} When $k$ is much larger than $s$.  Let
$$\{b_1, \dots, b_k\}=\underbrace{k+s^2, k+s^2, k+s^2}_{3}, \underbrace{s-1,s-1}_{2}, \underbrace{s-2, s-2}_{2},\dots,\underbrace{2, 2}_{2}, \underbrace{1,\dots, 1}_{k-2s+1}\},
$$ with $k+s^2>s-1>\dots>1$. By Theorem~\ref{thm2}, $q(x)$ is uniquely maximized by
$x=(k+s^2, k+s^2, 1,\dots, 1, 2, 2, 3, 3, \dots, s-1, s-1, k+s^2)$. Hence Conjecture~\ref{con2} is not true for this case.

%\begin{center}
%{\bf Acknowledgements}
%\end{center}

 %The authors would like to thank the  anonymous referees very much  for
 % valuable suggestions, corrections and comments which results in  a great improvement of
%the original manuscript. In particular, one
 % referee gives a clear and short construction of $T^*$ and provides references
 % \cite{fischermann2005}, \cite{klein1997},
 % \cite{polansky1986}, \cite{rouvray2002} and \cite{ruch1979}.


\frenchspacing
\begin{thebibliography}{}




\bibitem{bondy1976} J.~A.~Bondy and U.~S.~R.~Murty, {\it Graph theory with
applications,}  Macmillan Press, New York, 1976.


\bibitem{cela2011}E.~Cela, N.~Schmuck, S.~Wimer and G.~Woeginger, The Wiener maximum quadratic assignment problem, {\it Discrete Optimization} 8(2011) 411-416.

\bibitem{collins1989}K.~L.~Collins, On a conjecture of a Graham and Lov\'{a}sz about distance matrices, {\it Discrete Applied Mathematics,} 25(1989) 27-35.


%\bibitem{dankelmann1994}P.~Dankelmann, Average distance and
%independence number, {\it Discrete Applied Mathematics}, 51(1994)
%75-83.

\bibitem{dobrynin2001}A.~A.~Dobrynin, R.~Entringer and I.~Gutman,
Wiener index of trees: theory and applications, {\it Acta Appl.
Math.,} 66(2001) 211-249.

%\bibitem{entringer1976}R.~C.~Entringer, D.~E.~Jackson and
%D.~A.~Synder, Distance in graphs, {\it Czechoslovak Math. J.}
%26(1976) 283-296.

%\bibitem{erdos1960}P.~Erd\"{o}s and T.~Gallai, Graphs with
%prescribed degrees of vertices (Hungarian) {\it Mat. Lapok},
%11(1960) 264-274.



\bibitem{fischermann2002}M.~Fischermann, A.~Hoffmann, D.~Rautenbach,
L.~Szekely and L.~Volkmann, Wiener index versus maximum degree in
trees, {\it Discrete Applied Mathematics,} 122(2002) 127-137.


%\bibitem{fischermann2005}M.~Fischermann, D.~Rautenbach and
%L.~Volkmann, Extremal trees with respect to dominance order, {\it
%Ars Comb.} 76(2005) 249-255.

\bibitem{graham1971}R.~L.Graham and H.~O.~Pollak, On the addressing problem for loop switching, {\it Bell. System Tech. J., } 50(1971) 2495-2519.

\bibitem{graham1978}R.~L.~Graham and L.~Lov\'{a}sz, Distance matrix polynomials of trees, {\it Advanced in Mathematics}, 29(1978) 60-88.

%\bibitem{gutman2000}I.~Gutman, W.~Linert, I.~Lukovits and
%\v{Z}~Tomovi\'{c}, The multilicative version of the Wiener index,
%{\it J. Chem. Inf. Comput. Sci.} 40(2000) 113-116.

%\bibitem{gutman1997}I.~Gutman and J.~H.~Potgieter, Wiener index and
%intermolecular forces, {\it J. Serb. Chem. Soc. } 62(1997) 185-192.

%\bibitem{gutman1993}I.~Gutman, Y.~N.~Yeh, S.~L.~Lee and J.~C.~Chen,
%Wiener numbers of dendrimers, {\it Comm. Math. Chem. (MATCH) }
%30(1994) 103-115.

%\bibitem{horn1985}R.~A.~Horn and C.~R.~Johnson, {\it Matrix
%Analysis,}  Cambridge University Press, London, 1985.


%\bibitem{hosoya1971}H.~Hosoya, Topological index, A newly proposed
%quantity characterizing the topological nature of structural isomers
%of saturated hydrocarbons, {\it Bull. Chem. Soc. Jpn.} 4(1971)
%2332-2339.

%\bibitem{klein1997}D.~J.~Klein and D.~Babi${\rm acute{c}}$, Partial
%orderings in chemistry, {\it J. Chem. Inf. and Comp. Sci.} 37(1997)
%656-671.


\bibitem{jelen2003}F.~Jelen and E.~Triesch, Superdominance order and
distance of trees with bounded maximum degree, {\it Discrete Applied
Mathematics}, 125(2003) 225-233.


\bibitem{liu2008}M.~liu and X.~F.~Pan, On the Wiener index of trees with fixed diamter,
{\it MATCH Commun. Math. Comput. Chem.,} 60(2008) 85-94.

%\bibitem{lovasz1993}L.~Lov\'{a}sz, {\it Combinatorial Problems and
%Exercises,} 2nd Edition, North-Holland, Amsterdam, 1993.

%\bibitem{lukovits1991}I.~Lukovits, General formulas for the Wiener
%index, {\it J.~Chem.~Inf.~Comput.~Sci.~} 31(1991) 503-507.

%\bibitem{Merris1994} R. Merris, Laplacian matrices of graphs: a
%survey, {\it Linear Algebra  and Applications,} 197-198 (1994):
%143-176.



%\bibitem{marshall1979}A.~W.~Marshall and I.~Olkin, {\it
%Inequalities: Theory of Majorization and Its Applications,}
%Mathematics in Science and engineering, vol.143, Academic Proess,
%New York, 1979.

%\bibitem{polansky1986}O.~E.~Polansky and D.~Bonchev, The Wiener number of graphs. I.
%general theory and changes due to some graph operations {\it MATCH
%Commun. Math. Comput. Chem.} 21(1986) 133-186.

%\bibitem{rouvray2002}D.~H.~Rouvray and R.~B.King, {\it Topology in
%Chemistry}, Horwood Pub., Chichester, 2002.


%\bibitem{ruch1979}E.~Ruch and I.~Gutman, The branching extent of
%graphs, {\it J. Comb. Inf. and System Sci.} 4(1979) 285-295.

\bibitem{ruzoeh1990}S.~N.~Ruzieh and D.~L.~Powers, The distance spectrum of the path $P_n$ and the first distance eigenvector of connected graphs, {\it Linear and Multilinear Algebra}, 28(1990) 75-81.

%\bibitem{schmuck2012} N.~Schmuck, S.~Wagner, H.~Wang, Greedy trees, caterpillars, and %Wiener-type graph invariants, {\it MATCH Commun. Math. Comput. Chem.} 68(2012) 273-292.

%\bibitem{shi1993}R.~Shi, The average  distance of trees, {\it
%Systems Science and Mathematical Sciences}, 6(1)(1993)18-24.


\bibitem{sills2012}A.~V.~Sills and H.~Wang, On the maximal Wiener index and related questions, {\it Discrete Applied Mathematics}, 160(2012) 1615-1623.





\bibitem{wang2008}H.~Wang, The  extremal values of the Wiener index
of a tree with given degree sequence, {\it Discrete Applied
Mathematics}, 156(2009) 2647-2654.

\bibitem{wang2010}H.~Wang, Corrigendum: the extremal values of the Wiener index of a tree with given degree sequence, {\it Discrete Applied Mathematics }, 157(2008)
3754.

\bibitem{wiener1947}H.~Wiener,  Structural determination of paraffin
boiling points, {\it J. Amer. Chem. Soc.},  69(1947) 17-20.



\bibitem{Zhang2008} X.-D.~Zhang, Q.-Y.~Xiang, L.-Q.~Xu and R.-Y.~Pan,
The Wiener index of trees with given degree sequences, {\it MATCH
Commun. Math. Comput. Chem.}, 60(2008) 623-644.

\bibitem{zhang2010}X.-D.~Zhang,  Y.~Liu and M.-X.~Han, Maximum Wiener index of trees with
given degree sequence, {\it  MATCH Commun. Math. Comput. Chem.},
64(2010) 661-682.

\end{thebibliography}
\end{document}